\documentclass[10pt,reqno]{amsart}
\usepackage{pdfsync}
\usepackage{amssymb,amscd}
\usepackage{color}

\theoremstyle{plain}
\newtheorem{theorem}{Theorem}[section]
\newtheorem{proposition}[theorem]{Proposition}
\newtheorem{corollary}[theorem]{Corollary}
\newtheorem{lemma}[theorem]{Lemma}

\theoremstyle{definition}

\newtheorem{remarks}[theorem]{Remarks}

\newtheorem{example}[theorem]{Example}

 \DeclareMathOperator{\Id}{\mathrm{Id}}
 \newcommand{\N}{\mathbb{N}}

 \renewcommand{\geq}{\geqslant}
 \renewcommand{\leq}{\leqslant}

\begin{document}
\title[Geometry of von Neumann algebra preduals]{On the Geometry of von Neumann algebra preduals}

\author[M.~Mart\'{\i}n]{Miguel Mart\'{\i}n$^1\,$}
\address[Mart\'{\i}n]{Departamento de An\'{a}lisis Matem\'{a}tico,
Facultad de Ciencias,
Universidad de Granada,
E-18071 Granada, Spain}
\email{\texttt{mmartins@ugr.es}}
\thanks{$^1\,$Supported by by Spanish MICINN and FEDER project no.\ MTM2009-07498, Junta de Andaluc\'{\i}a and FEDER grants FQM-185 and P09-FQM-4911, and by ``Programa Nacional de Movilidad de Recursos Humanos del Plan Nacional de I+D+i 2008--2011'' of the Spanish MECD}

\author[Y.~Ueda]{Yoshimichi Ueda$^2\,$}
\address[Ueda]{Graduate School of Mathematics,
Kyushu University,
Fukuoka, 819-0395, Japan}
\email{\texttt{ueda@math.kyushu-u.ac.jp}}
\thanks{$^2\,$Supported by Grant-in-Aid for Scientific Research (C) 20540214.}


\begin{abstract}
Let $M$ be a von Neumann algebra and let $M_\star$ be its (unique) predual. We study when for every $\varphi\in M_\star$ there exists $\psi\in M_\star$ solving the equation $\|\varphi \pm \psi\|=\|\varphi\|=\|\psi\|$. This is the case when $M$ does not contain type I nor type III$_1$ factors as direct summands and it is false at least for the unique hyperfinite type III$_1$ factor. We also characterize this property in terms of the existence of centrally symmetric curves in the unit sphere of $M_\star$ of length $4$. An approximate result valid for all diffuse von Neumann algebras allows to show that the equation has solution for every element in the ultraproduct of preduals of diffuse von Neumann algebras and, in particular, the dual von Neumann algebra of such ultraproduct is diffuse. This shows that the Daugavet property and the uniform Daugavet property are equivalent for preduals of von Neumann algebras.
\end{abstract}

\date{October 6th, 2013}

\maketitle
\allowdisplaybreaks{

\section{Introduction}

Let $M$ be an arbitrary von Neumann algebra and denote its unique predual by $M_\star$, and they are known to be non-commutative counterparts to usual $L_\infty(\mu)$ and $L_1(\mu)$, respectively, with measure $\mu$. We will investigate the geometry of $M_\star$, and hence our investigation in the present paper should be regarded as analysis on non-commutative $L_1$-spaces. What we actually want to investigate is whether or not for every $\varphi \in M_\star$ there exists $\psi \in M_\star$ so that
\begin{equation}\label{eq:main}\tag{$\diamondsuit$}
\Vert \varphi \pm \psi\Vert = \Vert\varphi\Vert = \Vert\psi\Vert.
\end{equation}
It is easy to show that this is the case when $M_\star$ is $L_1[0,1]$ (see \cite[Example 5.12]{GuerreroMartin:JFA05} for instance), and moreover that if equation \eqref{eq:main} is solvable, then $M$ must be diffuse. Here the diffuseness is equivalent to that $\mu$ is atomless when $M = L_\infty(\mu)$ or $M_\star = L_1(\mu)$. The necessary definitions and background on von Neumann algebras are included in section~\ref{sec:VNA}.

For an element $x$ in an arbitrary Banach space $X$, the existence of $y\in X$ solving the equation $\|x \pm y\|=\|x\|=\|y\|$ is equivalent to the fact that $x$ is the center of a segment in the closed ball of radius $\|x\|$ of maximal length (i.e., $2\|x\|$). Indeed, if $\|x \pm y\|=\|x\|=\|y\|$, then $x=\tfrac12 (x+y) + \tfrac12 (x-y)$ with $\|(x+y)-(x-y)\|=2\|x\|$; conversely, if $x=\tfrac12 x_1 + \tfrac12 x_2$ with $\|x\|=\|x_1\|=\|x_2\|=\tfrac12 \|x_1-x_2\|$, then $y=\tfrac12 (x_1-x_2)$ satisfies $\|x \pm y\|=\|x\|=\|y\|$. Therefore, what we are investigating is when every point in a  closed ball of the predual of a von Neumann algebra fails to be an extreme point of the ball in the strongest possible way. Here we recall that the fact that the balls in the predual of a diffuse von Neumann algebra lack to have extreme points is well-known. This discussion together with the fact that the result is immediate for $L_1[0,1]$, and the investigation on the (uniform) Daugavet property in \cite[\S5]{GuerreroMartin:JFA05} motivate us to study equation \eqref{eq:main}.

In the present paper we will see that if neither type I nor type III$_1$ factor appears as a direct summand of a von Neumann algebra $M$, then for every element $\varphi\in M_\star$ equation \eqref{eq:main} has a solution $\psi\in M_\star$. Moreover, we will clarify which $\varphi$ has such $\psi$ in general situation, showing that $\psi$ is $\varphi e_+ - \varphi e_-$ for some pair $e_\pm$ of orthogonal projections in the centralizer of $|\varphi|$ (see the comment just before Lemma~\ref{L3} for the definition of centralizers). As a consequence of the latter, the answer to the question is negative at least for the unique hyperfinite type III$_1$ factor, though it is diffuse. This is an exactly non-commutative phenomenon that does never appears in the commutative setting. These results can be interpreted as follows: if a diffuse von Neumann algebra $M$ does not contain any factor of type III$_1$ as a direct summand, then every element $\varphi$ in the unit sphere of $M_\star$ can be connected to $-\varphi$ through a curve lying on the unit sphere with length $2$; this does not happen for the unique hyperfinite type III$_1$ factor.
On the other hand, for every diffuse $M$ we will also provide the following approximate result: for any given $\varphi \in M_\star$ and any given $\varepsilon>0$, there exists $\psi_\varepsilon \in M_\star$ such that $\Vert\varphi\Vert-\varepsilon < \Vert\varphi\pm\psi_\varepsilon\Vert < \Vert\varphi\Vert+\varepsilon$ and $\Vert\psi_\varepsilon\Vert = \Vert\varphi\Vert$. This approximate property turns out to be a characterization of the diffuseness.  

Next, we use the approximate result mentioned above to show that for every sequence of diffuse von Neumann algebras $M_i$, $i\in \N$, and every free ultrafilter on $\N$, equation \eqref{eq:main} is solvable in the ultraproduct $([M_i]_\star)/\mathcal{U}$. In particular, it shows that the dual of any ultraproduct of preduals of diffuse von Neumann algebras (which is known to be a von Neumann algebra, see \cite{Groh:JOT84,Raynaud:JOT02}) must be diffuse. Moreover, every nonzero normal positive linear functional on such dual von Neumann algebra has diffuse centralizer.

These results have an application to the so-called (uniform) Daugavet property.  We give a bit detailed explanation on the property for the reader who is not familiar with it. A Banach space $X$ is said to have the \emph{Daugavet property} (see \cite{KSSW}) if every rank-one bounded linear operator $T:X\longrightarrow X$ satisfies the norm identity $\|\Id + T\| = 1 + \|T\|$, in which case, all weakly compact bounded linear operators on $X$ also satisfy the same norm identity. The basic examples of Banach spaces satisfying the Daugavet property are $C(K)$ for perfect $K$ and $L_1(\mu)$ for atomless $\mu$. The non-commutative counterparts to this results were given by T.~Oikhberg \cite{Oik}, proving that the Daugavet property of a $C^*$-algebra is equivalent to its diffuseness and that the predual of a von Neumann algebra has the Daugavet property if and only if the algebra does (i.e., the algebra is diffuse). Latter, it was proved in \cite[\S 4]{GuerreroMartin:JFA05} that the predual of a von Neumann algebra has the Daugavet property if and only if its closed unit ball has no extreme points.

In \cite{BKSW} a stronger version of the Daugavet property called the \emph{uniform Daugavet property} was introduced and it was seen to be equivalent to the fact that every ultrapower of the space has the Daugavet property. The basic examples of spaces having the uniform Daugavet property are again $C(K)$ for perfect $K$ and $L_1(\mu)$ for atomless $\mu$, but there exists a Banach space with the Daugavet property which fails the uniform version, see \cite{KaWe}. The non-commutative counterparts to these positive examples are also true: on the one hand, it is proved in \cite[Theorem~5.2]{GuerreroMartin:JFA05} that a diffuse $C^*$-algebra has the uniform Daugavet property; on the other hand, it was claimed in \cite[Theorem~5.6]{GuerreroMartin:JFA05} that the predual of a diffuse von Neumann algebra has the uniform Daugavet property, but the proof there has a serious problem which the authors do not know how to solve. Our results above show that the dual von Neumann algebra to the ultrapower of the predual of any diffuse von Neumann algebra must be diffuse, and the ultrapower itself turns out to have the Daugavet property. Therefore, we provide an alternative and correct proof of \cite[Theorem~5.6]{GuerreroMartin:JFA05} for preduals of von Neumann algebras.

\section{Norm Equation $\Vert\varphi\pm\psi\Vert = \Vert\varphi\Vert = \Vert\psi\Vert$ in unknown $\psi$} \label{sec:VNA}

Our first goal is to show that in the predual of a diffuse von
Neumann algebra containing no type III$_1$ factor as a direct
summand, the equation entitling the section has solution for
every $\varphi$.

We need to fix notation and to present some preliminary
results. In what follows very basic terminologies (positivity
for linear functionals, etc.) on $C^*$-algebras will be used
freely. Throughout this section, let $M$ be a von Neumann
algebra or $W^*$-algebra, i.e., a (unital) $C^*$-algebra with a (unique
isometric) predual $M_\star$ (Sakai's space-free formulation,
see \cite[Theorem 1.16.7]{Sakai:Book} for its
justification). \emph{We do never assume the separability of $M_\star$.} A bounded linear functional on $M$ is
\emph{normal} if it falls in $M_\star$ or, in other words, if
it is weak$^*$--continuous. A nonzero projection $p \in M$ is
\emph{minimal} if $pMp = \mathbb{C}p$. When no minimal
projection exists, we say that $M$ is \emph{diffuse}. The
\emph{center} of $M$ is denoted by $\mathcal{Z}(M)$ and it is the commutative von Neumann subalgebra of $M$ consisting of those elements that commute with every element of $M$. When
$\mathcal{Z}(M) = \mathbb{C}1$, $M$ is said to be a
\emph{factor}. The next fact is probably well-known, but we
cannot find a suitable reference.

\begin{lemma}\label{L1} There is a unique orthogonal family
$\{z_i\}_{i\in I}$ of minimal projections in $\mathcal{Z}(M)$ such that $\mathcal{Z}(M)z_0$ is diffuse if $z_0 := 1-\sum_{i\in I}z_i \neq 0$. Then $M = Mz_0 \oplus \sum_{i\in I}^\oplus Mz_i$ {\rm(}$\ell^\infty$-direct sum{\rm)} and $\mathcal{Z}(Mz_i) = \mathcal{Z}(M)z_i$ holds for every $i \in \{0\}\sqcup I$. In particular, $Mz_0$ has diffuse center and every $Mz_i$, $i \in I$, becomes a factor.
\end{lemma}

\begin{proof} Choose a maximal family $\{z_i\}_{i\in I}$
of mutually orthogonal, nonzero projections in $\mathcal{Z}(M)$
such that $\mathcal{Z}(M)z_i = \mathbb{C}z_i$.
Set $z_0 := 1 - \sum_{i\in I}z_i$. The maximality of $\{z_i\}_{i\in I}$ forces either $z_0 = 0$ or that $\mathcal{Z}(M)z_0$ must be diffuse.
Let $e \in \mathcal{Z}(M)$ be an arbitrary minimal projection. Since $z_i e = e z_i$ is a projection, the minimality of $e$ forces that there is a unique $i(e) \in I$ with $e = z_{i(e)} e = z_{i(e)}$.
It immediately follows that the family $\{z_i\}_{i\in I}$ is uniquely determined. The rest is immediate (by \cite[Proposition 2.2.1]{Sakai:Book} etc.).
\end{proof}

The factors are classified into those of \emph{type I}
(possessing minimal projections), those of \emph{type II} (no
minimal projection but nonzero normal tracial positive linear
functionals exist) and those of \emph{type III} (no minimal
projection and no nonzero normal tracial positive linear
functional). One easily observes (cf.~the proof of Lemma~\ref{L3} below) that $M$ is diffuse if and only if every
$Mz_i$, $i \in I$, is not of type I in the notations of Lemma~\ref{L1}. The type III factors are further classified into the subclasses of \emph{type III$_\lambda$}, $0 \leq \lambda \leq
1$ based on the so-called modular theory. Although this finer
III$_\lambda$--classification theory due to Connes plays a key
r\^{o}le in what follows, we do not review it and refer to
\cite{Connes:ASENS73},\cite[Ch.~XII]{Takesaki:Book2} instead.
The next lemma is well-known (see e.g.~\cite[Lemma 11, Lemma
12]{Ueda:MathScand01}), but we do give its detailed proof with
explicit references for the reader's convenience. Here a
positive $\varphi \in M_\star$ is said to be \emph{faithful} if
$\varphi(x^* x) = 0$ implies $x=0$ for $x \in M$. In what
follows, the \emph{central support} (in $M$) of a projection $p
\in M$, i.e., the smallest projection $z \in \mathcal{Z}(M)$
with $z \geq p$, is denoted by $c_M(p)$.

\begin{lemma}\label{L2} If $M$ is a factor and has
a faithful positive $\varphi\in M_\star$ whose centralizer
$M_\varphi := \{x \in M\,|\, \text{$\varphi(xy) =
\varphi(yx)$ for all $y\in M$}\}$ is not diffuse,
then $M$ must be of either type I or type III$_1$.
\end{lemma}

\begin{proof} It suffices to prove that if $M$ is a factor of either type II or type III$_\lambda$ with $\lambda \neq 1$, then the centralizer $M_\varphi$ of any faithful positive $\varphi \in M_\star$ must be diffuse.

Firstly, assume that $M$ is of type II. By \cite[Theorem VIII.3.14, Theorem VIII.2.11]{Takesaki:Book2} one can find a positive selfadjoint (possibly unbounded) operator $h$ affiliated with $M$ so that $\sigma_t^\varphi = \mathrm{Ad}h^{it}$, $t \in \mathbb{R}$, and hence $M_\varphi = \{x \in M\,|\,\text{$xh^{it} = h^{it}x$ for every $t\in\mathbb{R}$}\}$. On contrary, suppose that there exists a minimal projection $e$ in $M_\varphi$. Choose a MASA $A$ in $M$ that contains $\{h^{it}\,|\,t\in\mathbb{R}\}\cup\{e\}$. Clearly $A$ sits in $M_\varphi$. Then $Ae \subseteq eM_\varphi e = \mathbb{C}e$ so that $Ae = \mathbb{C}e$, that is, $e$ is minimal in $A$. Since $A$ is a MASA in $M$, $eMe$ falls in $A$, implying that $e$ is minimal in $M$, a contradiction. Consequently, $M_\varphi$ must be diffuse.

Secondly, assume that $M$ is of type III$_\lambda$ with $0 \leq \lambda < 1$. By \cite[Theorem 4.2.1 (a); Theorem 5.2.1 (a)]{Connes:ASENS73} there exists a MASA $A$ in $M$ such that $A \subset M_\varphi \subset M$. On contrary, suppose that $M_\varphi$ has a minimal projection, say $e$, that is $eM_\varphi e = \mathbb{C}e$. By \cite[Proposition 6.4.3]{KadisonRingrose:Book2} $M_\varphi z$ with $z := c_{M_\varphi}(e)$ is a factor which possesses a minimal projection $e$ and a finite faithful normal trace $\varphi\!\upharpoonright_{M_\varphi z}$; hence it must be a type I$_n$ factor with $n$ finite (or a finite dimensional factor) by \cite[Corollary 6.5.3, Remark 6.5.4]{KadisonRingrose:Book2}. Note that $z$ falls in $A$, since $A$ is a MASA in $M_\varphi$. Hence $Az$ must be a MASA in $M_\varphi z$ so that $Az = \sum_{i=1}^n \mathbb{C}e_i$ with minimal projections $e_i$'s by \cite[Exercise 6.9.23]{KadisonRingrose:Book2} (or \cite[Lemma 3.7]{Kadison:AmerJMath84}). Therefore, $e_i Me_i$ must fall in $A$, implying $e_i M e_i = \mathbb{C}e_i$ a contradiction since $M$ is of type III. Consequently, $M_\varphi$ must be diffuse.
\end{proof}

The \emph{centralizer} $M_\varphi$ of a given (not necessarily faithful) nonzero positive $\varphi \in M_\star$ is defined to be the usual one (see Lemma~\ref{L2}) of the faithful $\varphi\!\upharpoonright_{s(\varphi)Ms(\varphi)}$, where $s(\varphi)$ denotes the \emph{support} of $\varphi$, i.e., $1-p$ of the greatest projection $p \in M$ with $\varphi(p) = 0$. It is easy to see that $\varphi(x) = \varphi(s(\varphi)x) = \varphi(x s(\varphi))$ for $x \in M$ and that a positive $\varphi \in M_\star$ is faithful if and only if $s(\varphi) = 1$. Note that this definition of centralizers is probably not standard.

\begin{lemma}\label{L3} If $\mathcal{Z}(M)$ is diffuse or if $M$ is a factor of neither type I nor type III$_1$, then the centralizer $M_\varphi$ of any nonzero positive $\varphi \in M_\star$ becomes diffuse.
\end{lemma}
\begin{proof} By \cite[Proposition 2.2.11]{Sakai:Book} one has $\mathcal{Z}(s(\varphi)Ms(\varphi)) = \mathcal{Z}(M)s(\varphi)$, and it is easy to see, by the definition of $c_M(s(\varphi))$, that $x \mapsto xs(\varphi)$ gives an injective $*$-homomorphism from $\mathcal{Z}(M)c_M(s(\varphi))$ onto $\mathcal{Z}(M)s(\varphi)$. This and \cite[Corollary 3.2.8]{Connes:ASENS73} show that $s(\varphi)Ms(\varphi)$ is again either a von Neumann algebra with diffuse center or a factor of neither type I nor type III$_1$. Hence we may and do assume that $s(\varphi)=1$, that is, $\varphi$ is faithful, by replacing $M$ by $s(\varphi)Ms(\varphi)$.

Firstly, assume that $\mathcal{Z}(M)$ is diffuse. On contrary, suppose that $M_\varphi$ has a minimal projection, say $e$. By \cite[Proposition 6.4.3]{KadisonRingrose:Book2} $M_\varphi z$ with $z := c_{M_\varphi}(e)$ must be a finite dimensional factor (see the final paragraph in the proof of Lemma \ref{L2}), a contradiction since $\mathbb{C}z \neq \mathcal{Z}(M)z \subseteq \mathcal{Z}(M_\varphi z)$. Therefore, $M_\varphi$ has no minimal projection.

Secondly, assume that $M$ is a factor of neither type I nor type III$_1$. Since $\varphi$ is faithful, the desired assertion immediately follows from Lemma~\ref{L2}.
\end{proof}

For $x \in M$ and $\varphi \in M_\star$ we define $x\varphi, \varphi x \in M_\star$ by $(x\varphi)(y) := \varphi(yx)$ and $(\varphi x)(y) := \varphi(xy)$ for $y \in M$. It is known, see \cite[\S1.14]{Sakai:Book}, that any $\varphi \in M_\star$ admits a unique \emph{polar decomposition} $\varphi = v|\varphi|$, where $|\varphi| \in M_\star$ is positive and $v$ is a partial
isometry in $M$ with $v^* v = s(|\varphi|)$. The core idea of the proof of the next proposition comes from the sketch given
to \cite[Lemma 3.6]{Ueda:BLMS11} there.

\begin{theorem}\label{T1}
Let $M$ be a von Neumann algebra and $M_\star$ be its unique
predual. Consider the decomposition $M = Mz_0 \oplus \sum_{i\in
I}^\oplus Mz_i$ as given in Lemma~\ref{L1}. If every $Mz_i$, $i
\in I$, is of neither type I nor type III$_1$, then every $\varphi \in M_\star$ has $\psi \in M_\star$ so that
$\Vert\varphi\pm\psi\Vert = \Vert\varphi\Vert = \Vert\psi\Vert$.
\end{theorem}

\begin{proof} Let $M = Mz_0 \oplus \sum_{i\in I}^\oplus Mz_i$ be as in Lemma~\ref{L1}. Accordingly one can decompose $\varphi = \varphi_0 + \sum_{i\in I} \varphi_i$ with $\varphi_i := \varphi\!\upharpoonright_{Mz_i}$, $i\in\{0\}\sqcup I$, and trivially $\Vert\varphi\Vert = \Vert \varphi_0\Vert + \sum_{i\in I} \Vert\varphi_i\Vert$ holds. (In this way, $M_\star$ is identified with $(Mz_0)_\star\oplus\sum_{i\in I}^\oplus(Mz_i)_\star$ ($\ell^1$-direct sum).) Consequently, it suffices to find $\psi_i \in (Mz_i)_\star$ in such a way that $\Vert \varphi_i \pm \psi_i\Vert = \Vert \varphi_i \Vert = \Vert \psi_i\Vert$ for each $i \in \{0\}\sqcup I$. Hence we will prove the desired assertion for nonzero $\varphi$ when $\mathcal{Z}(M)$ is diffuse or when $M$ is a factor of neither type I nor type III$_1$.

Let $\varphi = v|\varphi|$ be the polar decomposition, and by Lemma~\ref{L3} $M_{|\varphi|}$ must be diffuse. Choose a MASA $A$ in $M_{|\varphi|}$ with unit $s(|\varphi|) \in A$. If there existed a minimal projection $e$ in $A$, then $eM_{|\varphi|}e$ would sit in $A$ so that $eM_{|\varphi|}e = \mathbb{C}e$, a contradiction. Hence $A$ must be diffuse. Consider the faithful normal state $\chi := |\varphi|(s(|\varphi|))^{-1}|\varphi|\!\upharpoonright_A$ on $A$. Since $A$ is diffuse, one can construct a unital von Neumann subalgebra $C$ of $A$ in such a way that $(C,\chi\!\upharpoonright_C) \cong (L_\infty[0,1],\text{Leb})$ (see e.g.~\cite[Lemma 4.14]{Ueda:arXiv:1203.1806}). Hence one can easily find (by looking at $(L^\infty[0,1],\text{Leb})$) two orthogonal projections $e_\pm \in M_{|\varphi|}$ in such a way that $e_+ + e_- = s(|\varphi|)$ and $|\varphi|(e_\pm) = |\varphi|(s(|\varphi|))/2$. Consider the self-adjoint unitary $u := e_+ - e_- \in M_{|\varphi|}$, and $\psi := vu|\varphi| \in M_\star$, which clearly becomes a polar decomposition, since $(vu)^* vu = s(|\varphi|)$. Hence $\Vert\psi\Vert = |\varphi|(s(|\varphi|)) = \Vert\varphi\Vert$. Since $\varphi \pm \psi = v(s(|\varphi|)\pm u)|\varphi| = ve_\pm(2e_\pm|\varphi|)$ and the $e_\pm$'s falls in $M_{|\varphi|}$, we get $|\varphi\pm\psi| = 2e_\pm|\varphi|$ by the uniqueness of polar decompositions (see \cite[Theorem 1.14.4]{Sakai:Book}). Therefore, $\Vert\varphi\pm\psi\Vert  = 2|\varphi|(e_\pm) = 2(|\varphi|(s(|\varphi|))/2) = \Vert\varphi\Vert$.
\end{proof}

It is natural to ask whether or not the conclusion of Theorem~\ref{T1} still holds true when type III$_1$ factors appear as direct summands. As a consequence of the next theorem, we will see that it is not always so.

As usual, for a given $\varphi \in M_\star$ its adjoint $\varphi^* \in M_\star$ is defined by $\varphi^*(x) := \overline{\varphi(x^*)}$ for $x \in M$.

\begin{theorem}\label{T2} Let $M$ be a von Neumann algebra,
$M_\star$ be its unique predual, and $\varphi \in M_\star$ be nonzero.
Then the following are equivalent{\rm:}
\begin{itemize}
\item[(i)] $\varphi$ is the center of a segment of
length $2\Vert\varphi\Vert$ in the closed ball of radius
$\Vert\varphi\Vert$.
\item[(ii)] There is $\psi \in M_\star$ such that
$\Vert\varphi\pm\psi\Vert = \Vert\varphi\Vert = \Vert\psi\Vert$.
\item[(iii)] There are two orthogonal projections $e_\pm \in M_{|\varphi|}$ such that $e_+ + e_- = s(|\varphi|)$, $|\varphi|(e_\pm) = |\varphi|(s(|\varphi|))/2 = \Vert\varphi\Vert/2$. If this is the case, then $\psi$ can be chosen to be $v(e_+ - e_-)|\varphi| = \varphi e_+ - \varphi e_-$,  where $\varphi=v|\varphi|$ is the polar decomposition of $\varphi$.
\end{itemize}
In particular, for a given $\varphi \in M_\star$, the solutions $\psi \in M_\star$ of the norm equation $\Vert\varphi\pm\psi\Vert = \Vert\varphi\Vert = \Vert\psi\Vert$ are given by $\psi = \varphi e_+ - \varphi e_-$ with orthogonal projections $e_\pm$ in $M_{|\varphi|}$ satisfying that $e_+ + e_- = s(|\varphi|)$ and $|\varphi|(e_\pm) = \Vert\varphi\Vert/2$.
\end{theorem}

We need the following lemma which follows from a result due to
Kusuda \cite{Kusuda:PRIMS95}.

\begin{lemma}\label{L4} Let $\varphi_1, \varphi_2 \in M_\star$ be given. If $\Vert\varphi_1 + \varphi_2\Vert = \Vert\varphi_1\Vert + \Vert\varphi_2\Vert$ holds, then the polar decomposition $\varphi_1 + \varphi_2 = v|\varphi_1+\varphi_2|$ satisfies that $|\varphi_1+\varphi_2| = |\varphi_1| + |\varphi_2|$ and $\varphi_i = v|\varphi_i|$, $i=1,2$.
\end{lemma}

\begin{proof} By \cite[Theorem 2.1 (2)]{Kusuda:PRIMS95} one has $|\varphi_1+\varphi_2| = |\varphi_1| + |\varphi_2|$. By \cite[Theorem 2.1 (4)]{Kusuda:PRIMS95} one has $\varphi_i(v^*) = \Vert\varphi_i\Vert$ for both $i=1,2$. Then the proof of \cite[Theorem 1.14.4]{Sakai:Book} shows that $\varphi_i = v|\varphi_i|$ for both $i=1,2$. \end{proof}

\begin{proof}[Proof of Theorem~\ref{T2}] (i) $\Leftrightarrow$ (ii) is a general fact for arbitrary Banach spaces, as pointed out in the introduction.

(iii) $\Rightarrow$ (ii) follows from (the final part of) the proof of Theorem~\ref{T1}. Thus we do not repeat the same argument here.

(ii) $\Rightarrow$ (iii): Set $\varphi_{\pm} :=
\frac{1}{2}(\varphi\pm\psi) \in M_\star$, and then $\varphi =
\varphi_+ + \varphi_-$ and $\psi = \varphi_+ - \varphi_-$. By assumption, one has $\Vert\varphi\Vert = \Vert\varphi_+ +
\varphi_-\Vert \leq \Vert\varphi_+\Vert + \Vert\varphi_-\Vert = \frac{1}{2}\Vert\varphi+\psi\Vert +
\frac{1}{2}\Vert\varphi-\psi\Vert = \Vert\varphi\Vert$,
and hence $\Vert\varphi_+ + \varphi_-\Vert =
\Vert\varphi_+\Vert + \Vert\varphi_-\Vert$. Let $\varphi =
v|\varphi|$ be the polar decomposition of $\varphi$. Lemma~\ref{L4} shows that $|\varphi| = |\varphi_+| + |\varphi_-|$ and
moreover that $\varphi_\pm = v|\varphi_\pm|$. In particular, we
get $\chi:= v^*\psi = v^*(\varphi_+ - \varphi_-) = |\varphi_+|
- |\varphi_-|$, and thus $\chi^* = \chi$. With the polar
decomposition $\chi = w|\chi|$ one has $|\chi|v^*v = w^*\chi
v^*v = w^*\chi^* v^*v = w^*\psi^* vv^* v = w^*\psi^* v =
|\chi|$, implying $s(|\chi|) \leq v^* v$. Since $vv^*
\varphi_\pm = vv^* v|\varphi_\pm| = \varphi_\pm$, one has $\psi
= \varphi_+ - \varphi_- = vv^*\varphi_+ - vv^*\varphi_- =
vv^*\psi = v\chi$, implying $\chi \neq 0$ due to
$\Vert\psi\Vert=\Vert\varphi\Vert \neq 0$. By \cite[Theorem
III.4.2 (ii)]{Takesaki:Book1} (or the proof of \cite[Theorem 1.14.3]{Sakai:Book}) there are two orthogonal
projections $e_\pm \in M_{|\chi|}$ such that $e_+ + e_- =
s(|\chi|)$, $\pm(e_\pm\chi) \geq 0$, and $\chi = (e_+ -
e_-)|\chi|$ (the polar decomposition of $\chi = v^*\psi$).
Since $s(|\chi|) \leq v^* v$ and $\psi = v\chi = v(e_+ -
e_-)|\chi|$, the uniqueness of polar decompositions shows that
$|\chi| = |\psi|$. By assumption one also has $\Vert \psi\Vert
= \Vert \varphi_+ - \varphi_-\Vert \leq \Vert\varphi_+\Vert + \Vert\varphi_-\Vert = \frac{1}{2}\Vert\varphi
+ \psi\Vert + \frac{1}{2}\Vert\varphi-\psi\Vert =
\Vert\psi\Vert$, and hence $\Vert \varphi_+ - \varphi_-\Vert =
\Vert \varphi_+\Vert + \Vert\varphi_-\Vert$. Lemma~\ref{L4}
again shows that $|\psi| = |\varphi_+ + (-\varphi_-)| =
|\varphi_+| + |-\varphi_-| = |\varphi_+| + |\varphi_-| =
|\varphi|$. Consequently, $\psi = v (e_+ - e_-)|\varphi|$
becomes the polar decomposition of $\psi$. Then $\varphi\pm\psi
= v(2e_\pm|\varphi|)$, and $2|\varphi|(e_\pm) =
\Vert\varphi\pm\psi\Vert = \Vert\varphi\Vert =
|\varphi|(s(|\varphi|))$.
\end{proof}

We may now provide an example showing that Theorem~\ref{T1} cannot be extended to all diffuse von Neumann algebras.

\begin{example}\label{example}
{\slshape Consider $M$ to be the unique hyperfinite type III$_1$ factor. Then, there exists a nonzero $\varphi\in M_\star$ for which the equation $\|\varphi \pm \psi\|=\|\varphi\|=\|\psi\|$ has no solution $\psi\in M_\star$.} Indeed, by \cite[\S3]{HermanTakesaki:CMP70} (see also \cite[p.246--247]{NeshveyevStormer:EntropyBook}), there is a faithful normal state $\varphi$ in $M_\star$ with trivial centralizer, and thus Theorem~\ref{T2} tells us that the equation indeed has no solution $\psi$ for $\varphi$.
\end{example}

The remaining question is apparently whether the norm equation entitling this section is always solvable or not in the predual of a given type III$_1$ factor. It seems a very non-trivial question and we have known that it is certainly negative for several type III$_1$ factors including, as presented above, the unique hyperfinite type III$_1$ factor.

Next, we do give an approximate variant of Theorem~\ref{T1} that in turn holds even for arbitrary diffuse von Neumann algebras. The key is a deep result due to Haagerup and St\o rmer \cite{HaagerupStormer:Adv90}.

\begin{proposition}\label{P2} Let $M$ be a diffuse von Neumann algebra and $M_\star$ be its unique predual. Then for every $\varphi\in M_\star$ and every $\varepsilon>0$, there is $\psi_\varepsilon \in M_\star$ so that $\Vert\varphi\Vert - \varepsilon < \Vert \varphi\pm\psi_\varepsilon\Vert < \Vert\varphi\Vert + \varepsilon$ and $\Vert\psi_\varepsilon\Vert = \Vert\varphi\Vert$.
\end{proposition}

\begin{proof} By the proof of Theorem~\ref{T1}, it suffices to assume that $M$ is a type III$_1$ factor, since the desired assertion holds without error $\varepsilon>0$ in the other cases and since only countably many $\varphi_i$'s can be nonzero. We may and do also assume that $\varphi$ is nonzero.

Let $\varphi = v|\varphi|$ be the polar decomposition. Applying
\cite[Theorem 11.1]{HaagerupStormer:Adv90} to $c\,|\varphi|\!\upharpoonright_{s(|\varphi|)Ms(|\varphi|)}$ with $c := 1/\Vert\,|\varphi|\,\Vert = 1/\Vert\varphi\Vert$,  one can find a state $\eta_\varepsilon \in M_\star$ and $u_\varepsilon \in M$ in such a way that
$$
s(\eta_\varepsilon)=s(|\varphi|) = u_\varepsilon^* u_\varepsilon = u_\varepsilon u_\varepsilon^*,
\qquad
\big\Vert u_\varepsilon^* \eta_\varepsilon u_\varepsilon - c\,|\varphi| \big\Vert
=
\big\Vert \eta_\varepsilon - u_\varepsilon(c\,|\varphi|)u_\varepsilon^*\big\Vert < c\,\varepsilon,
$$
and $M_{\eta_\varepsilon}$ is of type II. Note here that any reduced algebra of a type III$_1$ factor by a nonzero projection becomes again of type III$_1$ thanks to
\cite[Corollary 3.2.8]{Connes:ASENS73}. Define a positive $\chi_\varepsilon := c^{-1}\,u_\varepsilon^* \eta_\varepsilon u_\varepsilon \in M_\star$. Then $s(\chi_\varepsilon) = s(\eta_\varepsilon) = s(|\varphi|)$ and $\Vert\chi_\varepsilon - |\varphi|\Vert < \varepsilon$. Moreover, by \cite[Corollary VIII.1.4]{Takesaki:Book2} one has $M_{\chi_\varepsilon} = u_\varepsilon M_{\eta_\varepsilon}u_\varepsilon^*$, being of type II, i.e., diffuse. Set
$\varphi_\varepsilon := v\chi_\varepsilon$, being a polar
decomposition. By the proof of Theorem~\ref{T1}, each
$\varphi_\varepsilon$ has $\psi_\varepsilon \in M_\star$ so
that $\Vert\varphi_\varepsilon \pm \psi_\varepsilon\Vert =
\Vert\varphi_\varepsilon\Vert = \Vert \psi_\varepsilon\Vert$.
By the construction of $\varphi_\varepsilon$, we observe that
$\Vert\varphi_\varepsilon\Vert = \Vert\chi_\varepsilon\Vert =
\Vert\varphi\Vert$ and $|\Vert \varphi\pm\psi_\varepsilon\Vert
- \Vert \varphi_\varepsilon\pm\psi_\varepsilon\Vert| \leq
\Vert\varphi-\varphi_\varepsilon\Vert = \Vert v(|\varphi| -
\chi_\varepsilon)\Vert \leq \Vert |\varphi| - \chi_\varepsilon
\Vert < \varepsilon$. Hence the assertion follows.
\end{proof}

Here are some remarks on centralizers of normal positive linear functionals.

\begin{remarks}\label{R2} Let $M$ be a von Neumann algebra.
{\slshape \begin{itemize}
\item[(1)] The following are equivalent{\rm:}
\begin{itemize}
\item[(a)] there exists a nonzero positive $\varphi \in M_\star$ with $M_\varphi = \mathbb{C}s(\varphi)$,
\item[(b)] there exists a nonzero positive $\varphi \in M_\star$ such that $M_\varphi$ is not diffuse.
\end{itemize}
\item[(2)] If $M_\varphi$ is diffuse for every nonzero positive $\varphi \in M_\star$, then $M$ itself is diffuse {\rm(}note that the reverse implication is false{\rm)}.
\end{itemize} }
\end{remarks}

\begin{proof}
(1): (a) $\Rightarrow$ (b) is trivial. (b) $\Rightarrow$ (a) is shown as follows. Assume that $\varphi$ is a nonzero normal positive linear functional on $M$ such that $M_\varphi$ has a minimal projection, say $e$. Then the nonzero normal positive linear functional $e\varphi e$ has the trivial centralizer $M_{e\varphi e} = eM_\varphi e = \mathbb{C}e$.

(2): If $M$ is not diffuse, then there is a minimal projection $e$ in $M$ so that $Mz$ with $z = c_M(e)$ becomes a type I factor by \cite[Proposition 6.4.3]{KadisonRingrose:Book2}, a contradiction.
\end{proof}

The next proposition sumarizes when the equation entitling the section is solvable for every $\varphi\in M_\star$. It includes a characterization in terms of paths in $M_\star$ which we will interpret in terms of the so-called flat spaces and girth curves.

\begin{proposition}\label{P3} Let $M$ be a von Neumann algebra, and $M_\star$ denotes its unique predual.
\begin{itemize}
\item[(1)] The following are equivalent{\rm:}
\begin{itemize}
\item[(i)] The norm equation $\Vert\varphi\pm\psi\Vert = \Vert\varphi\Vert = \Vert\psi\Vert$ has a solution $\psi \in M_\star$ for every $\varphi \in  M_\star$.
\item[(ii)] $M_\chi$ is diffuse for every non-zero, positive $\chi \in M_\star$.
\item[(iii)] For every $\varphi \in M_\star$ there is a path $\{\varphi_t\}_{0 \leq t \leq 2}$ in $M_\star$ such that $\varphi_0 = \varphi$, $\varphi_2 = -\varphi$, $\Vert\varphi_t\Vert = \Vert\varphi\Vert$ {\rm(}$0 \leq t \leq 2${\rm)}, and $\Vert\varphi_s - \varphi_t\Vert = \Vert\varphi\Vert\,|s-t|$ {\rm(}$0 \leq s, t \leq 2${\rm)}.
\end{itemize}
\item[(2)] If $M$ is diffuse, then the subset of all those $\varphi \in M_\star$ possessing a path $\{\varphi_t\}_{0 \leq t \leq 2}$ in $M_\star$ such that $\varphi_0 = \varphi$, $\varphi_2 = -\varphi$, $\Vert\varphi_t\Vert = \Vert\varphi\Vert$ {\rm(}$0 \leq t \leq 2${\rm)}, and $\Vert\varphi_s - \varphi_t\Vert = \Vert\varphi\Vert\,|s-t|$ {\rm(}$0 \leq s, t \leq 2${\rm)} is norm-dense in $M_\star$.
\end{itemize}
\end{proposition}

\begin{proof} (1): (iii) $\Rightarrow$ (i) is trivial; just $\psi := \varphi_1$.

(i) $\Rightarrow$ (ii): By Theorem~\ref{T2} the solvability of the equation for every $\varphi \in M_\star$ implies that $M_\chi \neq \mathbb{C}s(\chi)$ for every non-zero, positive $\chi \in M_\star$. By Remarks~\ref{R2}.(1) the latter is equivalent to  that $M_\chi$ is diffuse for every non-zero, positive $\chi \in M_\star$.

(ii) $\Rightarrow$ (iii): We may and do assume that a given $\varphi \in M_\star$ is non-zero. Let $\varphi = v|\varphi|$ be its polar decomposition. By assumption $M_{|\varphi|}$ must be diffuse. As in the proof of Theorem~\ref{T2}, one can choose a commutative von Neumann subalgebra $C$ of $M_{|\varphi|}$ with unit $s(|\varphi|)$ in such a way that $(C,|\varphi|(s(|\varphi|))^{-1}\,|\varphi|\!\upharpoonright_{C}) \cong (L_\infty[0,1],\mathrm{Leb})$. Let $e_t$, $1 \leq t \leq 2$, be the projection of $C$ that corresponds to the characteristic function $\chi_{[0,t/2]} \in L_\infty[0,1]$. Clearly
$$
|\varphi|(e_t - e_s) = (|\varphi|(s(|\varphi|))/2)(t-s) = (\Vert\varphi\Vert/2)(t-s) \qquad (0 \leq s \leq t \leq 2).
$$
Set $\varphi_t := v(s(|\varphi|)-2e_t)|\varphi|$, which becomes a unique polar decomposition, since $s(|\varphi|)-2e_t$ is a self-adjoint unitary in $M_{|\varphi|}$. Then one has $\varphi_0 = \varphi$, $\varphi_2 = - \varphi$, $\Vert\varphi_t\Vert = \Vert |\varphi|\Vert = \Vert\varphi\Vert$ and $\Vert\varphi_t - \varphi_s\Vert = \Vert 2 v(e_t - e_s)|\varphi|\Vert = 2|\varphi|(e_t - e_s) = \Vert\varphi\Vert\,(t-s)$ ($0 \leq s \leq t \leq 2$).

(2): The same argument as in the proof of Proposition~\ref{P2} shows that for every $\psi \in M_\star$ and every $\varepsilon > 0$ there is $\varphi_\varepsilon \in M_\star$ with $\|\varphi_\varepsilon\|=\|\varphi\|$ such that $M_{|\varphi_\varepsilon|}$ is diffuse and $\Vert\psi-\varphi_\varepsilon\Vert \leq \Vert |\psi| - |\varphi_\varepsilon|\Vert < \varepsilon$. The proof of (1) above shows that $\varphi_\varepsilon$ has the desired path. Hence the subset is norm-dense.
\end{proof}

Let us give an interpretation of the above in terms of flat spaces and girth curves. We refer the reader to the seminal papers \cite{Schaffer-paper,Harrell-Karlovitz} and the books \cite{Schaffer-book,vanDulst-book} for more information and background (though they are dealing with only real spaces). Given a (real or complex) Banach space $X$, we write $S_X$ to denote its unit sphere. The space $X$ is called \emph{flat} if there exists $x \in S_X$ and a (simple) curve lying on $S_X$ connecting $\pm x$ with length $2$. Equivalently, there is a path $t\longmapsto x_t$ from $[0,2]$ to $S_X$ such that $x_0=x$, $x_2=-x$ and $\Vert x_s - x_t\Vert = |s-t|$ for every $0 \leq s, t \leq 2$ (see \cite[p.196--197]{vanDulst-book}). Such a curve/path is called a \emph{girth curve}. The existence of girth curves is unusual and it is a purely infinite-dimensional phenomenon. See \cite[Eq.(17.6)]{vanDulst-book} (valid for complex spaces), which shows that the dual of any flat space must be non-separable.

It seems interesting to study whether or not for every $\varphi \in S_{M_\star}$ there is a girth curve connecting $\pm\varphi$ as long as $M$ is diffuse. The answer in the commutative case (i.e.~ $M_\star = L_1(\mu)$) was known to be true (see \cite[Example~1]{Harrell-Karlovitz} and \cite[\S1]{Schaffer-Israel-L-spaces}). The above proposition (together with Theorem~\ref{T1} and Example~\ref{example}) shows that this also is the case when $M$ is assumed to be diffuse as well as to have no type III$_1$ factor as a direct summand, and it is not the case for the unique hyperfinite type III$_1$ factor. Thus the problem is again about general type III$_1$ factors. For general diffuse von Neumann algebras we only may say that the set of points of the unit sphere which are the starting point of a girth curve is dense.
We compile all these results in the following corollary.

{\begin{corollary}\label{cor:Girth_curves}
Let $M$ be a diffuse von Neumann algebra, and $M_\star$ denotes its unique predual.
\begin{itemize}
\item[(1)] The set of elements in $S_{M_\star}$ which are the starting point of a girth curve is dense in $S_{M_\star}$. 
\item[(2)] Moreover, if $M$ does not contain any type III$_1$ factor as a direct summand, then the set in {\rm(1)} is actually the whole $S_{M_\star}$.
\item[(3)] On the other hand, the set in {\rm(1)} is not the whole $S_{M_\star}$ at least when $M$ is the unique hyperfinite type III$_1$ factor.
\end{itemize}
\end{corollary}

 \section{Applications to ultraproducts and the (uniform) Daugavet property} \label{sec:Daugavet}

Let us recall the notion of (Banach spaces) ultraproducts (see e.g.~\cite{He}). Let $\mathcal{U}$ be a free ultrafilter on $\N$, and
let $\{ X_i \}_{i\in \N}$ be a sequence of Banach spaces. We
can consider the $\ell_\infty$-sum of the family $[\oplus_{i\in \N} X_i]_{\ell_\infty}$ together with the closed subspace $N_{\mathcal U}$ of all $\{x_i\}_{i\in \N} \in [\oplus_{i\in
\N} X_i]_{\ell_\infty}$ with $\lim_{\mathcal U} \Vert x_i
\Vert=0$. The quotient space $\left[\oplus_{i\in \N}
X_i\right]_{\ell_\infty}/ N_{\mathcal U}$ is called the
\emph{ultraproduct} of the family $\{ X_i \}_{i\in \N}$
(relative to ${\mathcal U}$), and is denoted by
$(X_i)/{\mathcal U}$. Let $(x_i)$ stand for the element of
$(X_i)/{\mathcal U}$ containing a given family $\{ x_i \} \in
\left[\oplus_{i\in \N} X_i\right]_{\ell_\infty}$. It is easy to
check that $\Vert (x_i) \Vert = \lim_{\mathcal U} \Vert x_i
\Vert$. If all the $X_i$ are equal to the same Banach space
$X$, the ultraproduct of the family is called the
\emph{ultrapower} of $X$ (relative to ${\mathcal U}$) and usually denoted by $X/\mathcal{U}$.

The results obtained in the previous section have the next two  applications to ultraproducts.

\begin{corollary}\label{C1}
Let $M_i$, $i \in \mathbb{N}$, be a sequence of diffuse von Neumann
algebras, $[M_i]_\star$, $i \in \mathbb{N}$, be their unique
preduals, and $\mathcal{U}$ be a free ultrafilter on
$\mathbb{N}$. Then the ultraproduct
$([M_i]_\star)/\mathcal{U}$ has the property that for every $f
\in ([M_i]_\star)/\mathcal{U}$ there exists $g \in
([M_i]_\star)/\mathcal{U}$ such that $\Vert f\pm g\Vert = \Vert
f\Vert = \Vert g\Vert$. In particular, every
nonzero normal positive linear functional on the dual von Neumann algebra $\mathcal{M}$ of the ultraproduct $([M_i]_\star)/\mathcal{U}$ has diffuse centralizer and, in particular, $\mathcal{M}$ itself is diffuse.
\end{corollary}

\begin{proof} We may and do assume that $f$ is nonzero.
Let $\{\varphi_i\}_{i \in \mathbb{N}}$ be a representative of
the $f$. By Proposition~\ref{P2}, for each $i \in \mathbb{N}$
one can find $\psi_i \in (M_i)_\star$ in such a way that $\Vert\psi_i\Vert =
\Vert\varphi_i\Vert$ and
$\Vert\varphi_i\Vert - 1/i < \Vert\varphi_i\pm\psi_i\Vert <
\Vert\varphi_i\Vert + 1/i$. Set $g = (\psi_i) \in
([M_i]_\star)/\mathcal{U}$, and then one has
\begin{equation*}
\Vert g\Vert = \lim_{\mathcal{U}} \Vert \psi_i\Vert =
\lim_{\mathcal{U}} \Vert \varphi_i\Vert = \Vert f\Vert \qquad \text{and} \qquad
\Vert f\pm g\Vert =
\lim_{\mathcal{U}}\Vert\varphi_i\pm\psi_i\Vert =
\lim_{\mathcal{U}}\Vert\varphi_i\Vert = \Vert f\Vert.
\end{equation*}
Now, let us first recall that $\mathcal{M}$ is a von Neumann algebra \cite{Groh:JOT84,Raynaud:JOT02}. If $\mathcal{M}$ had a nonzero normal positive linear functional whose centralizer is not diffuse, then Theorem~\ref{T2} together with Remarks~\ref{R2}.(1) would show that for some $f \in ([M_i]_\star)/\mathcal{U}$ there is no $g \in ([M_i]_\star)/\mathcal{U}$ that satisfies the property established in the first part, a contradiction. Thus we have obtained the second part thanks to Remarks~\ref{R2}.(2).
\end{proof}

Remark that the above corollary shows that diffuseness contrasts with semifiniteness under the ultrapower and ultraproduct procedures, see \cite[\S1]{Raynaud:JOT02}.

By using \cite[Theorem 2.1]{Oik} we may give an interpretation of the result above in terms of the Daugavet property which, together with \cite[Corollary 6.5]{BKSW}, shows that the equivalence between the Daugavet and the uniform Daugavet properties for preduals of von Neumann algebras. This says that we have provided a new and correct proof of  \cite[Theorem~5.6]{GuerreroMartin:JFA05}. In fact, we have the following corollary: 

\begin{corollary}\label{C2} Let $M$ be a von Neumann algebra. The following conditions are equivalent{\rm:}
\begin{itemize} 
\item[(1)] $M$ is diffuse. 
\item[(2)] The predual $M_\star$ satisfies the consequence of Proposition \ref{P2}, that is, for every $\varphi\in M_\star$ and every $\varepsilon>0$, there is $\psi_\varepsilon \in M_\star$ so that $\Vert\varphi\Vert - \varepsilon < \Vert \varphi\pm\psi_\varepsilon\Vert < \Vert\varphi\Vert + \varepsilon$ and $\Vert\psi_\varepsilon\Vert = \Vert\varphi\Vert$. 
\item[(3)] The ultraproduct $M_\star/\mathcal{U}$ has the Daugavet property. 
\item[(4)] $M_\star$ has the uniform Daugavet property. 
\item[(5)] $M_\star$ has the Daugavet propety. 
\end{itemize}
\end{corollary}
\begin{proof} 
(1) $\Rightarrow$ (2) and (2) $\Rightarrow$ (3) are Proposition \ref{P2} and Corollary \ref{C1} with \cite[Theorem 2.1 (b)]{Oik}, respectively. (3) $\Leftrightarrow$ (4) is \cite[Corollary 6.5]{BKSW}. (4) $\Rightarrow$ (5) is trivial by definition. (5) $\Leftrightarrow$ (1) is again \cite[Theorem 2.1 (b)]{Oik}. 
\end{proof} 

Remark that the consequence of Proposition \ref{P2} actually charactrizes when a given von Neumann algebra is diffuse. Note that the equivalence (1) $\Leftrightarrow$ (2) can be shown without appealing to the (uniform) Daugavet property. In fact, the referee kindly informed us of a direct proof of (2) $\Rightarrow$ (1). 

\section*{Acknowledgment} We thank Professor Gilles Godefroy for his comments to the first version of this paper, which gave us a motivation to provide Proposition~\ref{P3} and Corollary~\ref{cor:Girth_curves}. We also thank the referee for his or her careful reading of this paper and for suggesting us to emphasize the equivalence (1) $\Leftrightarrow$ (2) in Corollary \ref{C2}.


\begin{thebibliography}{99}

\bibitem{Connes:ASENS73} A.~Connes,
Une classification des facteurs de type III,
{\it Ann.~Sci.~\'{E}cole Norm.~Sup.}~(4), {\bf 6} (1973), 133--252.

\bibitem{GuerreroMartin:JFA05} J.~Becerra-Guerrero and M.~Mart\'{i}n, The Daugavet property of $C^*$-algebras, $JB^*$-triples, and of their isometric preduals, {\it J.~Funct.~Anal.}, {\bf 224} (2005), 316--337.

\bibitem{BKSW} D.~Bilik, V.~Kadets, R.~V.~Shvidkoy, and
D.~Werner, Narrow operators and the Daugavet property for
ultraproduts, \emph{Positivity} \textbf{9} (2005), 45--62.

\bibitem{Groh:JOT84} U.~Groh,
Uniform ergodic theorems for identity preserving Schwarz maps on $W^*$-algebras, {\it J.~Oper. Theor.}, {\bf 11} (1984), 395--404.

\bibitem{HaagerupStormer:Adv90} U.~Haagerup and E.~St{\o}rmer,
Equivalence of normal states on von Neumann algebras and the flow of weights,  {\it Adv.~Math.}, {\bf 83} (1990), 180--262.

\bibitem{Harrell-Karlovitz} R.~E.~Harrell and L.~A.~Karlovitz, The geometry of flat Banach spaces, \emph{Trans. Amer. Math Soc.} \textbf{192} (1974), 209--218.

\bibitem{He} S.~Heinrich, Ultraproducts in Banach space
    theory, \emph{J.\ Reine Angew.\ Math.\ } \textbf{313}
    (1980), 72--104.

\bibitem{HermanTakesaki:CMP70} R.~H.~Herman and M.~Takesaki,
States and automorphism groups of operator algebras, {\it Comm.~Math.~Phys.} {\bf 19} (1970), 142--160.

\bibitem{KSSW} V.~M.~Kadets, R.~V.~Shvidkoy, G.~G.~Sirotkin,
    and D.~Werner, Banach spaces with the Daugavet property,
    \emph{Trans.\ Amer.\ Math.\ Soc.\ } \textbf{352} (2000),
    855--873.

\bibitem{KaWe} V.~Kadets and D.~Werner,
A Banach space with the Schur and the Daugavet property, \emph{Proc. Amer. Math. Soc.} \textbf{132} (2004), 1765--1773.

\bibitem{Kadison:AmerJMath84} R.~V.~Kadison,
Diagonalizing matrices,
{\it Amer.~J.~Math.}, {\bf 106} (1984), 1451--1468.

\bibitem{KadisonRingrose:Book2} R.~V.~Kadison and J.~Ringrose, {\it Fundamentals of the theory of operator algebras, Vol. II, Advanced theory}, Graduate Studies in Mathematics, {\bf 16}, American Mathematical Society, Providence, RI, 1997.

\bibitem{Kusuda:PRIMS95} M.~Kusuda, Norm additivity conditions for normal linear functionals on von Neumann algebras, {\it Publ.~Res.~Inst.~Math.~Sci.}, {\bf 31} (1995),  721--723.

\bibitem{NeshveyevStormer:EntropyBook} S.~Neshveyev and E.~St{\o}rmer,  {\it Dynamical Entropy in Operator Algebras}, Ergebnisse der Mathematik und ihrer Grenzgebiete. 3. Folge. A Series of Modern Surveys in Mathematics, 50. Springer-Verlag, Berlin, 2006.

\bibitem{Oik} T.~Oikhberg,
The Daugavet property of $C^*$-algebras and non-commutative
$L_p$-spaces, \emph{Positivity} \textbf{6} (2002), 59--73.

\bibitem{Raynaud:JOT02} Y.~Raynaud,
On ultrapowers of non commutative $L_p$ spaces, {\it J.~Operator Theory}, {\bf 48} (2002), 41--68.

\bibitem{Sakai:Book} S.~Sakai,
{\it $C^*$-algebras and $W^*$-algebras}, Classics in Mathematics, Springer 1998.

\bibitem{Schaffer-paper} J.~J.~Sch\"{a}ffer, Inner diameter, perimeter, and girth of spheres, \emph{Math. Ann.} \textbf{173} (1967), 59--79.

\bibitem{Schaffer-Israel-L-spaces}  J.~J.~Sch\"{a}ffer, On the geometry of spheres in L-spaces, \emph{Israel J. Math.} \textbf{10} (1971), 114--120.

\bibitem{Schaffer-book} J.~J.~Sch\"{a}ffer, \emph{Geometry of spheres in normed spaces}, Lecture Notes in Pure and Applied Mathematics 20, Marcel Dekker, 1976.

\bibitem{Takesaki:Book1}  M.~Takesaki,
{\it Theory of operator algebras, I},  Encyclopaedia of Mathematical Sciences, 124,
Operator Algebras and Non-commutative Geometry, 5, Springer-Verlag, Berlin, 2002.

\bibitem{Takesaki:Book2} M.~Takesaki,
{\it Theory of Operator Algebras, II}, Encyclopedia of Mathematical Sciences, 125,
Operator Algebras and Non-commutative Geometry, 6, Springer, Berlin, 2003.

\bibitem{Ueda:MathScand01} Y.~Ueda,
Remarks on free products with respect to non-tracial states, {\it Math.~Scand.}, {\bf 88} (2001), 111--125.

\bibitem{Ueda:BLMS11} Y.~Ueda,
On the predual of non-commutative $H^\infty$,
{\it Bull.~London Math.~Soc.}, {\bf 43} (2011), 886--896.

\bibitem{Ueda:arXiv:1203.1806} Y.~Ueda,
Some analysis on amalgamated free products of von Neumann algebras in non-tracial setting, arXiv:1203.1806.

\bibitem{vanDulst-book} D.~van Dulst, \emph{Reflexive and superreflexive Banach spaces}, Mathematical Centre Tracts, 102, Mathematisch Centrum, Amsterdam, 1978.

\end{thebibliography}
\end{document}